\documentclass[a4paper,12pt, pdf]{article}
\usepackage{cmap}
\usepackage[cp1251]{inputenc}
\usepackage[english]{babel}
\usepackage[left=2cm,right=2cm,top=2cm,bottom=2cm]{geometry}
\usepackage{amsmath, amsthm,amssymb, hyperref,enumerate}
\usepackage{setspace}
\usepackage[all]{xy}
\theoremstyle{plain}
\newtheorem{theorem}{Theorem}
\newtheorem{lemma}{Lemma}
\newtheorem{proposition}{Proposition}
\newtheorem{corollary}{Corollary}[theorem]

\theoremstyle{definition}
\newtheorem{example}{Example}

\newtheorem{definition}{Definition}

\begin{document}

\begin{center}\Large
\textbf{Fitting like subgroups  of finite groups}
\normalsize

\smallskip
V.\,I. Murashka and A.\,F. Vasil'ev

 \{mvimath@yandex.ru, formation56@mail.ru\}

Faculty of Mathematics and Technologies of Programming,

 Francisk Skorina Gomel State University,  Gomel 246019, Belarus\end{center}

\begin{abstract}
  In this paper the concept of $\mathbb{F}$-functorial of a finite group was introduced. These functorials have many properties of the Fitting subgroup of a soluble group and the generalized Fitting subgroup of a finite group.  It was shown that the set of all $\mathbb{F}$-functorials is a complete distributive lattice and the cardinality of this lattice is continuum. The sharp bounds on the generalized Fitting height of mutually permutable product of two subgroups were obtained.
\end{abstract}


\textbf{AMS}(2010). 20D25,  20F17,   20F19.

\section*{Introduction}

Throughout this paper, all groups are finite,   $G$, 1, $p$ and  $\mathfrak{X}$  always denote a finite group,  the unit group,   a prime  and a class of groups, respectively.

H. Fitting \cite{Fitting1938} showed that the product of normal nilpotent subgroups is again nilpotent. Hence every group $G$ has the largest normal nilpotent subgroup
 $\mathrm{F}(G)$. Now this subgroup is called the Fitting subgroup. The following properties of the Fitting subgroup are well known:

\begin{proposition}\label{prop0} Let  $G$ be a group. The following holds:

$(1)$ $f(\mathrm{F}(G))\subseteq \mathrm{F}(f(G))$ for every epimorphism $f: G\to G^*$.

$(2)$ $\mathrm{F}(N)\subseteq \mathrm{F}(G)$ for every $N\trianglelefteq G$.

$(3)$  If $G$ is soluble, then $C_G(\mathrm{F}(G))\subseteq\mathrm{F}(G)$.

$(4)$  $\mathrm{F}(G/\Phi(G))=\mathrm{F}(G)/\Phi(G)=\mathrm{ASoc}(G/\Phi(G))\leq\mathrm{Soc}(G/\Phi(G))$.

  \end{proposition}

The Fitting subgroup has a great influence on the structure of a finite soluble group. Analyzing proves of theorems about soluble groups which use the Fitting subgroup, one can note that they often use Properties $(1)$, $(2)$ and $(3)$. In the class of all groups the Fitting subgroup does not have Property $(3)$. Moreover there are infinite number of non-isomorphic groups with the trivial Fitting subgroup.

Recall that a group is called quasinilpotent if every its element induces an inner automorphism on every its chief factor. The generalized Fitting subgroup $\mathrm{F}^*(G)$ is the set of all elements which induce an inner automorphism on every chief factor of $G$ \cite[X, Definition 13.9]{19}. It is also known as the greatest normal quasinilpotent subgroup and has many properties of the Fitting subgroup of a soluble group.
The generalized Fitting subgroup \cite[X, Theorem 13.13]{19} can also be defined by
$$ \mathrm{F}^*(G)/\mathrm{F}(G)=\mathrm{Soc}(\mathrm{F}(G)C_G(\mathrm{F}(G))/\mathrm{F}(G)).$$


  Note that the generalized Fitting subgroup is non-trivial in every group but there are groups in which it coincides with the Frattini subgroup. That is why there is another generalization of the Fitting subgroup   $\tilde{\mathrm{F}}(G)$ introduced by P. Schmid \cite{Schmid1972} and L.A. Shemetkov
 \cite[Definition~7.5]{f4} and defined by
 $$ \Phi(G)\subseteq \tilde{\mathrm{F}}(G) \textrm{ and
  } \tilde{\mathrm{F}}(G)/\Phi(G)=\mathrm{Soc}(G/\Phi(G)).$$







P. F\"orster \cite{Foerster1985} showed that $\tilde{\mathrm{F}}(G)$ can be defined by
 $$ \Phi(G)\subseteq \tilde{\mathrm{F}}(G) \textrm{ and
  } \tilde{\mathrm{F}}(G)/\Phi(G)=\mathrm{F}^*(G/\Phi(G)).$$
D.A. Towers \cite{TOWERS2017} defined and studied analogues of $\mathrm{F}^*$ and $\mathrm{\tilde F}$ for  Lie algebras.

According to B.I. Plotkin \cite{Plotkin1973} a functorial   is a function $\theta$ which assigns to each group $G$ its characteristic subgroup $\theta(G)$   satisfying $f(\theta(G)) =
\theta(f(G)) $ for any isomorphism $f: G \to G^*$. 
He used this concept to study radicals and corresponding to them classes of  not necessary finite groups   in the sense of the following definition.

\begin{definition}
  A functorial $\gamma$ is called a   radical  if 
  
  $(F0)$ $\gamma(\gamma(G))=\gamma(G)$,
  
  $(F1)$   $f(\gamma(G))\subseteq \gamma(f(G))$ for every epimorphism $f: G\to G^*$,
  
  $(F2)$  $\gamma(N)\subseteq \gamma(G)$ for every $N\trianglelefteq G$.
   \end{definition}
  Nowadays these radicals are called  Plotkin radicals and also have applications in the ring theory \cite[p. 28]{Gardner2003}.

Let $\mathrm{F}: G\to \mathrm{F}(G)$, $\mathrm{F}^*: G\to \mathrm{F}^*(G)$ and $\mathrm{\tilde F}: G\to \mathrm{\tilde F}(G)$. Then $\mathrm{F}, \mathrm{F}^*$ and $ \mathrm{\tilde F}$ are functorials.
As was mentioned in \cite[p. 240]{Plotkin1973} it is rather important to study functorials that contain their centralizer.

In this paper we will use the concept of functorial to study the generalizations of the Fitting subgroup and their applications in the structural study of groups.


\section{The algebra of Fitting like functorials}


Recall \cite{Plotkin1973} some definitions about functorials.
Let $\gamma_1$ and $\gamma_2$ be functorials.
 The upper product $\gamma_1\star\gamma_2$ of $\gamma_1$ and $\gamma_2$ is defined by $(\gamma_1\star\gamma_2)(G)/\gamma_1(G)=\gamma_2(G/\gamma_1(G))$ and the lower product $\gamma_1\circ\gamma_2$ of $\gamma_1$ and $\gamma_2$ is defined by $(\gamma_1\circ\gamma_2)(G)=\gamma_2(\gamma_1(G))$;    functorial $\gamma$ is called (lower) \emph{idempotent} if $\gamma(\gamma(G))=\gamma(G)$ for every group $G$. For a functorial $\gamma$ the following subgroups are defined

\begin{center}
$\gamma^{(1)}(G)=\gamma(G)$, $\gamma^{(i+1)}(G)=\gamma(\gamma^{(i)}(G))$ and $\gamma^\infty(G)=\cap_{i\in\mathbb{N}}\gamma^{(i)}(G)$.
\end{center}

Note that $(\mathrm{F}^*)^\infty=\mathrm{F}^*$. According to \cite{Foerster1984} $\mathrm{\tilde F}^\infty\neq\mathrm{\tilde F}$.

Let $\{\gamma_i\mid i\in I\}$  be a set of functorials. Then $(\bigwedge_{i\in I}\gamma_i)(G)=\bigcap_{i\in I}\gamma_i(G)$ and $(\bigvee_{i\in I}\gamma_i)(G)=\langle\gamma_i(G)\mid i\in I\rangle$ are functorials.

Analyzing the properties of the Fitting subgroup of a soluble group and its generalizations in the class of all groups we introduce the following

\begin{definition}
Let $\mathfrak{X}$ be a normally hereditary homomorph.   We shall call a   functorial  $\gamma$ an $\mathbb{F}$-functorial in  $\mathfrak{X}$ if for every $\mathfrak{X}$-group $G$   it   satisfies:

$(F1)$ $f(\gamma(G))\subseteq \gamma(f(G))$ for every epimorphism $f: G\to G^*$.

$(F2)$ $\gamma(N)\subseteq \gamma(G)$ for every $N\trianglelefteq G$.

$(F3)$ $C_G(\gamma(G))\subseteq\gamma(G)$.

$(F4)$ $\gamma(G)/\Phi(G)\subseteq \mathrm{Soc}(G/\Phi(G))$.
       \end{definition}

If $\mathfrak{X}$ is the class of all groups, then  an $\mathbb{F}$-functorial in $\mathfrak{X}$ will be called an $\mathbb{F}$-functorial. Note that an $\mathbb{F}$-functorial is a Plotkin radical iff it is idempotent.

\begin{definition}
  We shall call a functorial $\varphi$ a Frattini functorial if it satisfies $(F1)$,   $(F2)$     and $\varphi(G)\subseteq\Phi(G)$.
\end{definition}

The main result of this section is

\begin{theorem}\label{lattice}
Let $ \mathcal{R}$ be the set of all $\mathbb{F}$-functorials.

$(a)$ $(\mathcal{R}, \vee, \wedge)$   is a complete distributive lattice,
  $\mathrm{F}^*$ and $\mathrm{\tilde{F}}$    are its the smallest and the largest  elements respectively.

$(b)$  $(\mathcal{R}, \circ)$ is a semigroup, $\mathrm{F}^*$ and  $\mathrm{\tilde{F}}^\infty$ are its zero  and  the largest idempotent elements respectively.

$(c)$ If $\gamma$ and $\varphi$ are an $\mathbb{F}$-functorial and a Frattini functorial respectively, then $\varphi\star\gamma$ is an $\mathbb{F}$-functorial and
    $\varphi\star\mathrm{\tilde F}=\mathrm{\tilde F}$.

$(d)$ The cardinality of $\mathcal{R}$ is continuum.
\end{theorem}

\begin{corollary}\label{lattice11}
  Let $\gamma$ be an $\mathbb{F}$-functorial and  $\varphi\in\{\gamma^{(i)}\mid i\in\mathbb{N}\}$. If $\gamma$ is an $($idempotent$)$ $\mathbb{F}$-functorial, then $\varphi$ is also an $($idempotent$)$ $\mathbb{F}$-functorial.
\end{corollary}

\begin{corollary}\label{lattice12}
  The cardinality of the set of all idempotent $\mathbb{F}$-functorials is continuum.
\end{corollary}

\begin{theorem}\label{lattice0}
 $\mathrm{F}$ is the unique $\mathbb{F}$-functorial in the class of all soluble groups.  \end{theorem}

We need the following propositions in our proves.

\begin{proposition}\label{FF}
Let $\gamma$ be a functorial.

$(1)$ If $\gamma$  satisfies $(F1)$ and $(F2)$, then $\gamma(G_1\times G_2)=\gamma(G_1)\times \gamma(G_2)$ for every groups $G_1$\,and\,$G_2$.

$(2)$ If $\gamma$ satisfies $(F2)$ and $(F3)$, then $\mathrm{F}^*(G)\subseteq\gamma(G)$.

$(3)$  If $\gamma$ satisfies $(F4)$, then  $\gamma(G)\leq\mathrm{\tilde{F}}(G)$ for every group $G$.

$(4)$  $\gamma^\infty$ is idempotent.
     \end{proposition}


\begin{proof}
 $(1)$ From $G_i\trianglelefteq G_1\times G_2$ it follows that $\gamma(G_i)\subseteq \gamma(G_1\times G_2)$ by $(F2)$ for $i\in\{1, 2\}$.
Note that $\gamma(G_1\times G_2)G_i/G_i\subseteq \gamma((G_1\times G_2)/G_i)=(\gamma(G_{\bar i})\times G_i)/G_i$ by $(F1)$ for $i\in\{1, 2\}$. Now
    $$\gamma(G_1\times G_2)\subseteq (\gamma(G_1\times G_2)G_1)\cap (\gamma(G_1\times G_2)G_2)\subseteq (\gamma(G_1)\times G_2)\cap (G_1\times\gamma(G_2))=\gamma(G_1)\times \gamma(G_2). $$
    Thus $\gamma(G_1\times G_2)=\gamma(G_1)\times \gamma(G_2)$.

 $(2)$  Note that $\gamma(1)=1$. Assume that there is a $p$-group $G$ with $\gamma(G)\neq\mathrm{F}^*(G)=G$. Suppose that $G$ is a minimal order group with this property. It means that $\gamma(G)<G$ and $\gamma(M)=\mathrm{F}^*(M)=M$ for every maximal subgroup $M$ of $G$. Since all maximal subgroups of a $p$-group are normal, by $(F2)$ $\gamma(G)$ contains  all maximal subgroups of $G$. If $G$ has at least two maximal subgroups, then $\gamma(G)=G$, a contradiction with $\gamma(G)<G$. Thus $G$ has a unique maximal subgroup. It means that $G$ is cyclic. Now   $G=C_G(\gamma(G))\subseteq\gamma(G)<G$ by $(F3)$, the contradiction.  Thus  $\gamma(G)=\mathrm{F}^*(G)=G$ for every $p$-group $G$.

  From $\mathrm{O}_p(G)=\gamma(\mathrm{O}_p(G))\subseteq\gamma(G)$ by $(F2)$ and $\mathrm{F}(G)=\times_{p\in\pi(G)}\mathrm{O}_p(G)$ for every group $G$ it follows that $\mathrm{F}(G)\subseteq\gamma(G)$.

Let show that $\gamma(G)=G$ for every quasinilpotent group $G$. By \cite[X, Theorem 13.8]{19} $G/\mathrm{F}(G)$ is a direct product of simple non-abelian groups. Let $H/\mathrm{F}(G)$ be one of them. Note that $\mathrm{F}(G)=\mathrm{F}(H)$ and $H\trianglelefteq G$. So $H$ is quasinilpotent. Now $H=D\mathrm{F}(H)$ where $[D, \mathrm{F}(H)]=1$, $D\cap \mathrm{F}(H)=\mathrm{Z}(D)$ and $D/\mathrm{Z}(D)$ is a simple group by \cite[X, Theorem 13.8]{19}. It means that $\gamma(H)\in\{\mathrm{F}(H), H\}$. Assume that $\gamma(H)=\mathrm{F}(H)$. Now $D\subseteq C_H(\mathrm{F}(H))=C_H(\gamma(H))\subseteq \gamma(H)$ by $(F3)$. Hence $\gamma(H)=H$, a contradiction with  $\gamma(H)=\mathrm{F}(H)$. Thus $\gamma(H)=H$. Now $G$ is a product of normal subgroups $H=\gamma(H)$. It means that $G=\gamma(G)$ by $(F2)$.

Note that $\mathrm{F}^*(G)=\gamma(\mathrm{F}^*(G))\subseteq\gamma(G)$ by $(F2)$.

$(3)$ According to $(F4)$ $\gamma(G)/\Phi(G)\subseteq \mathrm{Soc}(G/\Phi(G))=\mathrm{\tilde{F}}(G)/\Phi(G)$. Hence $\gamma(G)\leq\mathrm{\tilde{F}}(G)$ for every group $G$.

$(4)$ Since $G$ is a finite group, there is $n\in\mathbb{N}$ such that $\gamma^{(n)}(G)=\gamma^{(m)}(G)$ for all $m\geq n$. From $\gamma^{(i+1)}(G)\leq\gamma^{(i)}(G)$ for all $i\in\mathbb{N}$ it follows that $\gamma^\infty(G)=\cap_{i\in\mathbb{N}}\gamma^{(i)}(G)=\gamma^{(n)}(G)$. Now $$\gamma^\infty(\gamma^\infty(G))=\cap_{i\in\mathbb{N}}\gamma^{(i)}(\gamma^\infty(G))=
  \cap_{i\in\mathbb{N}}\gamma^{(i)}(\gamma^{(n)}(G))=\cap_{i\in\mathbb{N}, i\geq n+1}\gamma^{(i)}(G)=\gamma^{(n)}(G)=\gamma^\infty(G).$$
  Thus $\gamma^\infty$ is idempotent.
\end{proof}

Recall that a functorial is called hereditary if it satisfies

$(F5)$  $\gamma(G)\cap N\subseteq\gamma(N)$ for every $N\trianglelefteq G$.

Note that if $\gamma$ satisfies $(F2)$ and $(F5)$, then  $\gamma(G)\cap N=\gamma(N)$ for every $N\trianglelefteq G$.

\begin{proposition}\label{length0}
  Let $\gamma_1$ and $\gamma_2$ be functorials. If $\gamma_1$ and $\gamma_2$  satisfy $(F1)$ and $(F2)$, then $\gamma_2\star\gamma_1$ satisfies $(F1)$ and $(F2)$. Moreover if $\gamma_1$ and $\gamma_2$ also satisfy $(F5)$, then  $\gamma_2\star\gamma_1$ satisfies $(F5)$.
\end{proposition}

\begin{proof}
  $(1)$  \emph{$\gamma_2\star\gamma_1$ satisfies $(F1)$}.

  Let $f: G\to f(G)$ be an epimorphism.  From $f(\gamma_2(G))\subseteq\gamma_2(f(G))$ it follows that the following diagram is commutative.
\[
\xymatrix{
G \ar[r]^{f} \ar[dr]^{f_4} \ar[d]^{f_1} & f(G) \ar[d]^{f_3} \\
  G/\gamma_2(G) \ar[r]^{f_2} & f(G)/\gamma_2(f(G))
}
\]
Let $X=\gamma_1(G/\gamma_2(G))$ and $Y=\gamma_1(f(G)/\gamma_2(f(G)))$. Note that $(\gamma_2\star\gamma_1)(G)=f_1^{-1}(X)$ and $(\gamma_2\star\gamma_1)(f(G))=f_3^{-1}(Y)$ by the definition of $\gamma_2\star\gamma_1$.
Since $\gamma_1$ satisfies $(F1)$, we see that $f_2(X)\subseteq Y$.  Hence $X\subseteq f_2^{-1}(Y)$. Now $(\gamma_2\star\gamma_1)(G)\subseteq   f_1^{-1}(f_2^{-1}(Y))=f_4^{-1}(Y)$. So  $$f((\gamma_2\star\gamma_1)(G))\subseteq   f(f_4^{-1}(Y))=f_3^{-1}(Y)=(\gamma_2\star\gamma_1)(f(G)).$$ Thus $\gamma_2\star\gamma_1$ satisfies $(F1)$.

$(2)$ \emph{$\gamma_2\star\gamma_1$ satisfies $(F2)$}.

    Let $N\trianglelefteq G$. From $\gamma_2(N)\textrm{ char }N\trianglelefteq G$ it follows that $\gamma_2(N)\trianglelefteq G$.  Since $\gamma_2$ satisfies $(F2)$, we see that $\gamma_2(N)\subseteq \gamma_2(G)$. So the following diagram is commutative.

\[
\xymatrix{
G \ar[r]^{f_1} \ar[dr]_{f_3}& G/\gamma_2(N) \ar[d]^{f_2} \\
 & G/\gamma_2(G)
}
\]
Let $X=\gamma_1(G/\gamma_2(N))$, $Y=\gamma_1(N/\gamma_2(N))$ and $Z=\gamma_1(G/\gamma_2(G))$. Then      $(\gamma_2\star\gamma_1)(G)=f_3^{-1}(Z)$ and $(\gamma_1\star\gamma_2)(N)\subseteq f_1^{-1}(Y)$. Since $\gamma_1$ satisfies $(F1)$ and $(F2)$, we see that $f_2(X)\subseteq Z$ and $Y\subseteq X$. Now
     $$(\gamma_2\star\gamma_1)(N)\subseteq f_1^{-1}(Y)\subseteq f_1^{-1}(X)\subseteq f_1^{-1}(f_2^{-1}(Z))=f_3^{-1}(Z)=(\gamma_2\star\gamma_1)(G).$$
Hence $\gamma_2\star\gamma_1$ satisfies $(F2)$.

$ (3)$ \emph{If $\gamma_1$ and $\gamma_2$ also satisfy $(F5)$, then  $\gamma_2\star\gamma_1$ satisfies $(F5)$.}

 Assume that $\gamma_1$ and $\gamma_2$  satisfy $(F2)$ and $(F5)$. Let $N\trianglelefteq G$.

Since $N\gamma_1(G)/\gamma_1(G)\cap(\gamma_2\star\gamma_1)(G)/\gamma_1(G)\trianglelefteq (\gamma_2\star\gamma_1)(G)/\gamma_1(G)=\gamma_2(G/\gamma_1(G))$, we see that $$\gamma_2((N\gamma_1(G)\cap (\gamma_2\star\gamma_1)(G))/\gamma_1(G))=(N\gamma_1(G)\cap (\gamma_2\star\gamma_1)(G))/\gamma_1(G).$$
Note that
\begin{multline*}
  (N\gamma_1(G)\cap (\gamma_2\star\gamma_1)(G))/\gamma_1(G)=\\
 (N\cap (\gamma_2\star\gamma_1)(G))\gamma_1(G)/\gamma_1(G)\simeq
 (N\cap (\gamma_2\star\gamma_1)(G))/(N\cap\gamma_1(G))\\
=(N\cap (\gamma_2\star\gamma_1)(G))/\gamma_1(N)\trianglelefteq N/\gamma_1(N).
\end{multline*}
It means that $(N\cap (\gamma_2\star\gamma_1)(G))/\gamma_1(N)\subseteq \gamma_2(N/\gamma_1(N))$. Thus  $N\cap (\gamma_2\star\gamma_1)(G)\subseteq  (\gamma_2\star\gamma_1)(N)$, i.e $\gamma_2\star\gamma_1$ satisfies $(F5)$.\end{proof}


\begin{proof}[Proof of Theorem \ref{lattice}]
Let prove the statement $(a)$.

$(a.1)$ \emph{$\mathrm{F}^*$ is the smallest $\mathbb{F}$-functorial and idempotent.}

According to $(2)$  of Proposition \ref{FF} every $\mathbb{F}$-functorial contains   $\mathrm{F}^*$. From   \cite[X, Lemma 13.3(a), Corollary 13.11, Theorem 13.12]{19} it follows that $\mathrm{F}^*$ is idempotent and satisfies $(F1), (F2)$ and $(F3)$. Now $ \mathrm{F}^*(G)/\Phi(G)\subseteq \mathrm{ F}^*(G/\Phi(G))=\mathrm{\tilde F}(G)/\Phi(G)=\mathrm{Soc}(G/\Phi(G))$. So   $\mathrm{F}^*$ satisfies $(F4)$. Hence it is the smallest $\mathbb{F}$-functorial and idempotent.

$(a.2)$  \emph{$\mathrm{\tilde F}$ is the largest  $\mathbb{F}$-functorial and  non-idempotent.}

 According to $(3)$ of Proposition \ref{FF} every $\mathbb{F}$-functorial is contained in    $\mathrm{\tilde F}$. From \cite[Lemmas 1.3 and 1.4]{Foerster1985} it follows that $\mathrm{\tilde{F}}$ satisfies $(F1)$ and $(F2)$. According \cite[Theorem 7.12]{f4}  $\mathrm{\tilde{F}}$ satisfies $(F3)$. Note that $\mathrm{\tilde{F}}$ satisfies $(F4)$ by its definition. Hence it is an $\mathbb{F}$-functorial. The following example was suggested by P. F\"orster \cite{Foerster1984} and shows that $\mathrm{\tilde{F}}$ is not idempotent:

 Let $E\simeq \mathbb{A}_5$ be the symmetric group of degree 5. There is an $\mathbb{F}_5E$-module $V$ such
 that $R=Rad(V)$ is a faithful irreducible $\mathbb{F}_5E$-module and $V/R$ is
 an irreducible trivial $\mathbb{F}_5E$-module.  Let $H=V\leftthreetimes E$. Then $H'=RE$ is a primitive group  and $|H:H'|=5$. Now $\Phi(H)=R$ by \cite[B, Lemma 3.14]{s8}.
 By \cite[B, Theorem 10.3]{s8} there is an irreducible $\mathbb{F}_7H$-module $W$   with $C_H(W)=H'$.
 Let $G=W\leftthreetimes H$. Note that $W$ is a minimal normal subgroup of $G$.  Now every maximal subgroup of $G$ either contains $W$ and a subgroup isomorphic to a maximal subgroup of $H$ or is isomorphic to $H$. Since $R\trianglelefteq G$, we see that it is contained in both types of maximal subgroups of $G$.
 Thus $\Phi(G)=\Phi(H)=R$. Note that $\mathrm{Soc}(G/R)=(WR/R)\times (ER/R)$. So $\tilde{\mathrm{F}}(G)=W\times ER$
 and $\Phi(\tilde{\mathrm{F}}(G))=1$. It means that
 $\tilde{\mathrm{F}}(\tilde{\mathrm{F}}(G))=\mathrm{Soc}(\tilde{\mathrm{F}}(G))=R\times W<\tilde{\mathrm{F}}(G)$.

The module from this example can be constructed in GAP by the following commands.

This command creates a permutation $\mathbb{F}_5E$-module of dimension $5$:

  $K:=PermutationGModule(AlternatingGroup(5), GaloisField(5));$

This command finds the socle of the previous module:

  $L:=MTX.BasisSocle(K);$

This command finds the quotient module of our permutation module by its socle:

  $V:=MTX.InducedActionFactorModule(K, L);$

We can check that this module is indecomposable:

  $MTX.IsIndecomposable(V);$

Now we can find all its composition factors of it.

  $MTX.CompositionFactors(V);$

It has two composition factors. The dimension of them are 3 and 1.

 $(a.3)$ \emph{Let $\{\gamma_i\mid i\in I\}$  be a set of $\mathbb{F}$-functorials. Then $(\bigwedge_{i\in I}\gamma_i)(G)=\bigcap_{i\in I}\gamma_i(G)$ and $(\bigvee_{i\in I}\gamma_i)(G)=\langle\gamma_i(G)\mid i\in I\rangle$ are $\mathbb{F}$-functorials. In particular $(\mathcal{R},\vee,\wedge)$ is a complete lattice.}

 Let $f: G\to G^*$ be an epimorphism. Since every element of $\{\gamma_i\mid i\in I\}$  satisfies $(F1)$,    \begin{center}
 $f((\bigwedge_{i\in I}\gamma_i)(G))=f(\bigcap_{i\in I}\gamma_i(G))\subseteq \bigcap_{i\in I}f(\gamma_i(G))\subseteq \bigcap_{i\in I}\gamma_i(f(G))=(\bigwedge_{i\in I}\gamma_i)(f(G))$ and $f((\bigvee_{i\in I}\gamma_i)(G))=f(\langle\gamma_i(G)\mid i\in I\rangle)=\langle f(\gamma_i(G))\mid i\in I\rangle
 \subseteq \langle\gamma_i(f(G))\mid i\in I\rangle=(\bigvee_{i\in I}\gamma_i)(f(G))$.
\end{center}
It means that $\bigwedge_{i\in I}\gamma_i$ and $\bigvee_{i\in I}\gamma_i$ satisfy $(F1).$

Let $N\trianglelefteq G$. Then by $(F2)$
  \begin{center}
 $(\bigwedge_{i\in I}\gamma_i)(N)=\bigcap_{i\in I}\gamma_i(N)\subseteq \bigcap_{i\in I}\gamma_i(G)=(\bigwedge_{i\in I}\gamma_i)(G)$ and $(\bigvee_{i\in I}\gamma_i)(N)=\langle \gamma_i(N)\mid i\in I\rangle \subseteq\langle\gamma_i(G)\mid i\in I\rangle=(\bigvee_{i\in I}\gamma_i)(G)$.
\end{center}
Therefore $\bigwedge_{i\in I}\gamma_i$ and $\bigvee_{i\in I}\gamma_i$ satisfy $(F2).$

Since $\mathrm{F}^*(G)\subseteq\bigcap_{i\in I}\gamma_i(G)\subseteq \langle\gamma_i(G)\mid i\in I\rangle$ by Proposition \ref{FF}, we see that  $\bigwedge_{i\in I}\gamma_i$ and $\bigvee_{i\in I}\gamma_i$ satisfy $(F3).$

From $\gamma_i(G)/\Phi(G)\leq \mathrm{Soc}(G/\Phi(G))$, $i\in I$, it follows that

\begin{center}
$(\bigwedge_{i\in I}\gamma_i)(G)/\Phi(G)=\bigcap_{i\in I}(\gamma_i(G)/\Phi(G))\leq \mathrm{Soc}(G/\Phi(G)) $ and

$(\bigvee_{i\in I}\gamma_i)(G)/\Phi(G)=\langle \gamma_i(G)/\Phi(G)\mid i\in I\rangle\leq \mathrm{Soc}(G/\Phi(G)).$\end{center}
So $\bigwedge_{i\in I}\gamma_i$ and $\bigvee_{i\in I}\gamma_i$ satisfy $(F4).$

Thus $\bigwedge_{i\in I}\gamma_i$ and $\bigvee_{i\in I}\gamma_i$ are $\mathbb{F}$-functorials and $(\mathcal{R},\vee,\wedge)$ is a complete lattice.

$(a.4)$ \emph{Let $L(G)=\{\gamma(G)\mid\gamma$ is an $\mathbb{F}$-functorial$\}$. Then $L(G)$ is $($lattice$)$ isomorphic to a sublattice of a subset  lattice.  In particular,  $(\mathcal{R},\vee,\wedge)$  is a   distributive lattice.}

 Since $\mathrm{F}^*(G)\subseteq \gamma(G)$ for every $\mathbb{F}$-functorial $\gamma$, we see that $\overline{L}(G)=\{\gamma(G)/\mathrm{F}^*(G)\mid\gamma$ is an $\mathbb{F}$-functorial\} is isomorphic to $L(G)$. Then $\gamma(G)/\mathrm{F}^*(G)$ lies in the non-abelian socle $\overline{S}$ of $G/\mathrm{F}^*(G)$. Note that  $\overline{S}=\overline{N}_1\times\dots\times\overline{N}_n$ where $\overline{N}_i$ is a minimal normal subgroup of $G/\mathrm{F}^*(G)$.

  Let $f: L(G)\to\mathcal{P}(\{1,\dots, n\})$  where $f(\gamma(G))=\{i\mid\overline{N}_i\subseteq\gamma(G)/\mathrm{F}^*(G)\}$. Note that $f$ is injective.  According to \cite[X, Lemma 13.6(a)]{19} every normal in $G/\mathrm{F}^*(G)$ subgroup of $\overline{S}$  is a direct product of some $\overline{N}_i$. It means that
  $f((\gamma_1\wedge\gamma_2)(G))=f(\gamma_1(G)\cap \gamma_2(G))=f(\gamma_1(G))\wedge f(\gamma_2(G))$ and $f((\gamma_1\vee\gamma_2)(G))=f(\gamma_1(G)\gamma_2(G))=f(\gamma_1(G))\vee f(\gamma_2(G))$. So $f$ is a (lattice) isomorphism between $L(G)$ and some sublattice of   $\mathcal{P}(\{1,\dots, n\})$ which is distributive.

 Now the distributive laws hold for values of $\mathbb{F}$-functorials  for every group $G$. It means that   $(\mathcal{R},\vee,\wedge)$ is a distributive lattice.

Let prove the statement $(b)$.

  $(b.1)$ \emph{$(\mathcal{R}, \circ)$ is a semigroup and $\mathrm{F}^*$ is its zero element where  $(\gamma_2\circ\gamma_1)(G)=\gamma_1(\gamma_2(G))$.}

 To prove that $(\mathcal{R}, \circ)$ is a semigroup it is sufficient to prove that $\delta(G)=\gamma_1(\gamma_2(G))$ is an $\mathbb{F}$-functorial where $\gamma_1$ and $\gamma_2$ are $\mathbb{F}$-functorials.  This statement  will follow from the following 4 claims.

  $(b.1.1)$ \emph{If $\gamma_1$ and $\gamma_2$  satisfy $(F2)$, then $\delta$ also satisfies $(F2)$}.

According to $(F2)$
  $$\delta(N)=\gamma_1(\gamma_2(N))\subseteq \gamma_1(\gamma_2(G))=\delta(G).$$ Hence $\delta$ satisfies $(F2)$.


  $(b.1.2)$ \emph{If $\gamma_1$ and $\gamma_2$  satisfy $(F1)$ and $(F2)$, then $\delta$ also satisfies $(F1)$ and $(F2)$.}

By $(b.1.1)$ $\delta$ satisfies $(F2)$. Let show that it satisfies $(F1)$. Let $f$ be an epimorphism.       Note that $f(\gamma_i(G))\trianglelefteq \gamma_i(f(G))$  by $(F1)$, $i\in\{1, 2\}$. Now
  $$f(\delta(G))=f(\gamma_1(\gamma_2(G)))\subseteq\gamma_1(f(\gamma_2(G)))\subseteq \gamma_1(\gamma_2(f(G)))=\delta(f(G)).$$ Hence $\delta$ satisfies $(F1)$.

$(b.1.3)$  \emph{If $\gamma_1$ and $\gamma_2$  satisfy $(F2)$ and $(F3)$, then $\mathrm{F}^*(G)\subseteq\delta(G)$ for every group $G$. In particular, $\delta$     satisfies $(F3)$.}

From Proposition \ref{FF} it follows that  $\mathrm{F}^*(G)\subseteq\gamma_i(G)$ for $i\in\{1, 2\}$. Then $\mathrm{F}^*(G)=\gamma_1(\mathrm{F}^*(G))\subseteq\gamma_1(\gamma_2(G))=\delta(G)$.

$(b.1.4)$  \emph{If $\gamma_1$ and $\gamma_2$  satisfy $(F4)$, then  $\delta$  also   satisfies $(F4)$.}

Since $\gamma_2(G)/\Phi(G)\subseteq\mathrm{Soc}(G/\Phi(G))$, we see that $$\delta(G)/\Phi(G)=\gamma_1(\gamma_2(G))/\Phi(G)\subseteq\gamma_2(G)/\Phi(G)\subseteq\mathrm{Soc}(G/\Phi(G)).$$
Hence $\delta$ satisfies $(F4)$.

  Thus  $(\mathcal{R}, \circ)$ is a semigroup. Let $\gamma$ be an $\mathbb{F}$-functorial. Then $\mathrm{F}^*(\gamma(G))\subseteq\mathrm{F}^*(G)$. From $\mathrm{F}^*(G)\subseteq\gamma(G)\trianglelefteq G$ it follows that $\mathrm{F}^*(\gamma(G))=\mathrm{F}^*(G)$.    Note that
 $\mathrm{F}^*(G)=\mathrm{F}^*(\mathrm{F}^*(G))\subseteq\gamma(\mathrm{F}^*(G))\subseteq \mathrm{F}^*(G)$. Hence $\gamma(\mathrm{F}^*(G))=\mathrm{F}^*(G)$. It means that $\mathrm{F}^*\circ\gamma=\gamma\circ\mathrm{F}^*=\mathrm{F}^*$ for every $\mathbb{F}$-functorial $\gamma$. Thus $\mathrm{F}^*$ is a zero element of $(\mathcal{R}, \circ)$.

$(b.2)$ \emph{Let $\gamma$ be an $\mathbb{F}$-functorial and  $\delta\in\{\gamma^{(i)}\mid i\in\mathbb{N}\}\cup\{\gamma^\infty\}$. If $\gamma$ is an $($idempotent$)$ $\mathbb{F}$-functorial, then $\delta$ is also an $($idempotent$)$ $\mathbb{F}$-functorial.}

From $(b.1)$ it follows that if  $\delta\in\{\gamma^{(i)}\mid i\in\mathbb{N}\}$, then $\delta$ is an $\mathbb{F}$-functorial. Since $\mathcal{R}$ is a complete lattice, we see that $\delta^\infty$ is also an $\mathbb{F}$-functorial. Note that if $\gamma$ is idempotent, then $\gamma(G)=\gamma^{(1)}(G)=\gamma^{(2)}(G)=\dots=\gamma^{\infty}(G)$. Hence $\delta=\gamma$ is idempotent.

 $(b.3)$  \emph {$\mathrm{\tilde{F}}^\infty$ is the largest idempotent $\mathbb{F}$-functoral.}

From $(b.2)$ it follows that  $\mathrm{\tilde{F}}^\infty$ is an $\mathbb{F}$-functorial. From $(4)$ of Proposition \ref{FF} it follows that  $\mathrm{\tilde{F}}^\infty$ is idempotent.

   Assume that an $\mathbb{F}$-functorial $\gamma$ is idempotent.   Note that $\gamma(G)\subseteq\mathrm{\tilde{F}}(G)=\mathrm{\tilde{F}}^{(1)}(G)$. Assume now that  $\gamma^{(i)}(G)\subseteq\mathrm{\tilde{F}}^{(i)}(G)$. According to $(F2)$
  $$\gamma^{(i+1)}(G)=\gamma(\gamma^{(i)}(G))\subseteq\gamma(\mathrm{\tilde{F}}^{(i)}(G))\subseteq\mathrm{\tilde{F}}(\mathrm{\tilde{F}}^{(i)}(G))
  =\mathrm{\tilde{F}}^{(i+1)}(G). $$
  Since $\gamma$ is idempotent,   $\gamma(G)=\gamma^\infty(G)\subseteq \mathrm{\tilde{F}}^\infty(G)$.

$(c)$   \emph{If $\gamma$ and $\varphi$ are an $\mathbb{F}$-functorial and a Frattini functorial respectively, then $\varphi\star\gamma$ is an $\mathbb{F}$-functorial and $\varphi\star\mathrm{\tilde F}=\mathrm{\tilde F}$.}

$(c.1)$ $\varphi\star\gamma$ satisfies $(F1)$ and $(F2)$.

Directly follows from Proposition \ref{length0}.

$(c.2)$ $\varphi\star\mathrm{\tilde F}=\mathrm{\tilde F}$.

From $\varphi(G)\subseteq\Phi(G)$ it follows that  $\Phi(G/\varphi(G))=\Phi(G)/\varphi(G)$ and the following diagram is commutative:

\[
\xymatrix{
G \ar[r]^{f_1} \ar[dr]_{f_3}& G/\varphi(G) \ar[d]^{f_2} \\
 & G/\Phi(G)
}
\]

Let $X=\mathrm{Soc}(G/\Phi(G))$. Then $\mathrm{\tilde F}(G)=f_3^{-1}(X)$ and $\mathrm{\tilde F}(G/\varphi(G))=f_2^{-1}(X)$ by the definition of $\mathrm{\tilde F}$. Note that $(\varphi\star\mathrm{\tilde F})(G)=f_1^{-1}(\mathrm{\tilde F}(G/\varphi(G)))$ by its definition. Hence $(\varphi\star\mathrm{\tilde F})(G)=f_1^{-1}(f_2^{-1}(X))=f_3^{-1}(X)=\mathrm{\tilde F}(G)$. Thus $\varphi\star\mathrm{\tilde F}=\mathrm{\tilde F}$.

$(c.3)$  \emph{$\varphi\star\gamma$ is an $\mathbb{F}$-functorial.}

Note that $\mathrm{F}^*(G)/\varphi(G)\subseteq \mathrm{F}^*(G/\varphi(G))\subseteq\gamma(G/\varphi(G))$. Hence $\mathrm{F}^*(G)\subseteq(\varphi\star\gamma)(G)$. It means that $\varphi\star\gamma$ satisfies $(F3)$. From definition of $\star$ it follows that  if $\gamma_1(G)\subseteq\gamma_2(G)$ for every group $G$, then $(\varphi\star\gamma_1)(G)\subseteq(\varphi\star\gamma_2)(G)$ for every group $G$. Now $(\varphi\star\gamma)(G)\subseteq(\varphi\star\mathrm{\tilde F})(G)=\mathrm{\tilde F}(G)$. Hence $(\varphi\star\gamma)(G)/\Phi(G)\subseteq\mathrm{Soc}(G/\Phi(G))$, i.e. $\varphi\star\gamma$ satisfies $(F4)$. Thus   $\varphi\star\gamma$ is an $\mathbb{F}$-functorial.

$(d)$ \emph{The cardinality of $\mathcal{R}$ is continuum.}

According to  Cantor--Schr\"oder--Bernstein theorem to prove this statement it is enough to prove that there are injections from $\mathcal{R}$ to $\mathcal{P}(\mathbb{N})$ and from $\mathcal{P}(\mathbb{N})$ to $\mathcal{R}$.

$(d.1)$ \emph{There is an injection from $\mathcal{R}$ to $\mathcal{P}(\mathbb{N})$}.

  Let $\mathcal{G}$ be the set of non-isomorphic finite groups  such that for every finite group $H$ there is unique $G\in\mathcal{G}$ with $H\simeq G$. It is known that   $\mathcal{G}$ is  countable. From the definition of    functorial it follows that to define an $\mathbb{F}$-functorial $\gamma$ it is sufficient to   define it on each member of $\mathcal{G}$. For $H\in\mathcal{G}$ let $N(H)$ be the set of all characteristic subgroups of $H$. Note that $N(H)$ is a finite set. Let $\mathcal{N}$ be the disjoint union of  $N(H)$ for all $H\in\mathcal{G}$. Then $\mathcal{N}$ is countable  and every $\mathbb{F}$-functorial $\gamma$ is uniquely defined by the subset $\{\gamma(H)\mid H\in\mathcal{G}\}$ of $\mathcal{N}$. Hence there is an injection from   $\mathcal{R}$ to $\mathcal{P}(\mathcal{N})$. Since $\mathcal{N}$ is countable, there is a bijection from $\mathcal{P}(\mathcal{N})$ to $\mathcal{P}(\mathbb{N})$. Thus there is an injection from $\mathcal{R}$ to $\mathcal{P}(\mathbb{N})$.

$(d.2)$ \emph{There is an injection from $\mathcal{P}(\mathbb{N})$ to $\mathcal{R}$.}

 Recall that the set of all primes is denoted by $\mathbb{P}$.   Let $\pi\subseteq\mathbb{P}$ and $\Phi_\pi(G)=\mathrm{O}_\pi(\Phi(G))$ (if $\pi=\emptyset$, then  $\Phi_\pi(G)=1$). It is straightforward to check that $\Phi_\pi$ is a Frattini functorial.

  Let $\pi_1\neq\pi_2$ be a subsets of $\mathbb{P}$. WLOG we may assume that there is a prime $p\in\pi_1\setminus\pi_2$.   Let $G\simeq\mathbb{A}_p$ be the alternating group of degree $p$. Note that $p\in\pi(G)$.  From  \cite{Griess1978} it follows that there exists a faithful  $\mathbb{F}_pG$-module $A$
  such that  $A\rightarrow E\twoheadrightarrow G$ where $A\stackrel {G}{\simeq} \Phi(E)$ and
 $E/\Phi(E)\simeq G$.

 Note that $\mathrm{O}_{\pi_1}(\Phi(E))=\Phi(E)$, $\mathrm{O}_{\pi_2}(\Phi(E))=1$,  $\mathrm{F}^*(E)=\Phi(E)$ and $\mathrm{F}^*(E/\Phi(E))\simeq G$. It means that $(\Phi_{\pi_2}\star\mathrm{F}^*)(E)=\Phi(E)\neq E=(\Phi_{\pi_1}\star\mathrm{F}^*)(E)$. Hence $\Phi_{\pi_2}\star\mathrm{F}^*\neq \Phi_{\pi_1}\star\mathrm{F}^*$.

So there is an injection from $\mathcal{P}(\mathbb{P})$ to $\mathcal{R}$. Since   $\mathbb{P}$   is countable, we see that there is a bijection between  $\mathcal{P}(\mathbb{N})$ and $\mathcal{P}(\mathbb{P})$. Thus there is an injection from $\mathcal{P}(\mathbb{N})$ to $\mathcal{R}$.
  \end{proof}

\begin{proof}[Proof of Corollary \ref{lattice11}]
  Directly follows from $(b.2)$ of the proof of Theorem \ref{lattice}.
\end{proof}

\begin{proof}[Proof of Corollary \ref{lattice12}]
  Note that the set of all idempotent $\mathbb{F}$-functorials is a subset of $\mathcal{R}$.
  According to $(4)$ of Proposition \ref{FF} and the proof of Theorem \ref{lattice}  $(\Phi_{\pi}\star\mathrm{F}^*)^\infty$ is an idempotent $\mathbb{F}$-functorial.  Note that in $(d.2)$ of the proof of Theorem \ref{lattice} $(\Phi_{\pi_2}\star\mathrm{F}^*)((\Phi_{\pi_2}\star\mathrm{F}^*)(E))=\Phi(E)\neq E=(\Phi_{\pi_1}\star\mathrm{F}^*)((\Phi_{\pi_1}\star\mathrm{F}^*)(E))$. So $(\Phi_{\pi_2}\star\mathrm{F}^*)^\infty(E))=\Phi(E)\neq E=(\Phi_{\pi_1}\star\mathrm{F}^*)^\infty(E)$. Hence for $\pi_1\neq\pi_2$ functorials $(\Phi_{\pi_1}\star\mathrm{F}^*)^\infty$ and $(\Phi_{\pi_2}\star\mathrm{F}^*)^\infty$ are different. So the  set of all idempotent $\mathbb{F}$-functorials contains a subset of cardinality continuum. Thus its cardinality is continuum.
\end{proof}

\begin{proof}[Proof of Theorem \ref{lattice0}]
  Note that $\mathrm{F}^*(G)=\mathrm{\tilde F}(G)=\mathrm{F}(G)$ in a soluble group $G$. Now the statement of  Theorem \ref{lattice0} directly follows from $(2)$ and $(3)$ of Proposition \ref{FF}.
\end{proof}


\section{$\mathbb{F}$-functorials and classes of groups}

Recall that  a class of groups $\mathfrak{X}$ is called

$(1)$ \emph{$Q$-closed} or \emph{homomorph} if  every homomorphic image of an $\mathfrak{X}$-group is an $\mathfrak{X}$-group;

$(2)$ \emph{formation} if it is $Q$-closed and\,$G/(M\cap N)\in \mathfrak{X}$ when\,$G /M\in\mathfrak{X}$\,and\,$G/ N\in\mathfrak{X}$;

$(3)$ \emph{$N_0$-closed} if from $G=NM$ where $N$ and $M$ are normal $\mathfrak{X}$-subgroups of $G$ it follows that $G\in\mathfrak{X}$.

If $\mathfrak{X}$ is an $N_0$-closed class of groups with 1, then $G_\mathfrak{X}=\langle H\trianglelefteq G\mid H\in\mathfrak{X}\rangle$ is the largest normal $\mathfrak{X}$-subgroup of $G$.

\begin{proposition}
  Let $\gamma$ be an $\mathbb{F}$-functorial and $\mathfrak{F}(\gamma)=(G\mid\gamma(G)=G)$. Then $\mathfrak{F}(\gamma)=\mathfrak{F}(\gamma^\infty)$ is a $\{Q, N_0\}$-closed class of groups and $G_{\mathfrak{F}(\gamma)}=\gamma^\infty(G)$.
\end{proposition}

\begin{proof} Since $\gamma$ satisfies $(F1)$ and $(F2)$, we see that $\mathfrak{F}(\gamma)$ is  a $\{Q, N_0\}$-closed class of groups.
  Note that $\gamma(G)=G$ iff $\gamma(\gamma(G))=G$. Hence  $\gamma(G)=G$ iff $\gamma^\infty(G)=G$. So $\mathfrak{F}(\gamma)=\mathfrak{F}(\gamma^\infty)$.  Now $G_{\mathfrak{F}(\gamma)}=\langle H\trianglelefteq G\mid H\in\mathfrak{F}(\gamma)\rangle=\langle H\trianglelefteq G\mid H=\gamma^\infty(H)\rangle\subseteq\gamma^\infty(G)$ by $(F2)$. From $\gamma^\infty(\gamma^\infty(G))=\gamma^\infty(G)$ it follows that $\gamma^\infty(G)\in\mathfrak{F}(\gamma)$. Hence $\gamma^\infty(G)\subseteq G_{\mathfrak{F}(\gamma)}$. Thus $\gamma^\infty(G)= G_{\mathfrak{F}(\gamma)}$.
\end{proof}






Recall that a formation $\mathfrak{F}$ is called \emph{saturated} if from $G/\Phi(G)\in\mathfrak{F}$ always follows $G\in\mathfrak{F}$.

\begin{theorem}\label{trad} Let $\mathfrak{F}$ be an $N_0$-closed formation.

$(1)$ If $\mathfrak{F}$ is saturated and  $\mathrm{F}(G)\subseteq G_\mathfrak{F}\subseteq \tilde{\mathrm{F}}(G)$ holds for every group $G$, then $\mathfrak{F}=\mathfrak{N}$ is a formation of all nilpotent groups.

$(2)$ If $\mathrm{F}^*(G)\subseteq G_\mathfrak{F}\subseteq \tilde{\mathrm{F}}(G)$ holds for every group $G$,
 then $\mathfrak{F}=\mathfrak{N}^*$ is a formation of all quasinilpotent groups.
\end{theorem}

\begin{proof}
(1)
 From $\mathrm{F}(G)\subseteq G_\mathfrak{F}$ for every group $G$ it follows that $\mathfrak{N}\subseteq \mathfrak{F}$. Assume that
 $\mathfrak{F}\setminus\mathfrak{N}\neq\emptyset$. Let chose a minimal order group $G$ from
 $\mathfrak{F}\setminus\mathfrak{N}$. Since $\mathfrak{F}$ and $\mathfrak{N}$ are saturated formations,
 we see that    $\Phi(G)=1$ and $G$ has a unique minimal normal subgroup. From $\mathrm{F}(G)\subseteq G_\mathfrak{F}\subseteq \tilde{\mathrm{F}}(G)$ it follows that  $G=\mathrm{Soc}(G)$ is a non-abelian group. From  \cite{Griess1978} it follows that there exists a faithful  $\mathbb{F}_pG$-module $A$ for a $p\in\pi(G)$,
  such that  $A\rightarrow E\twoheadrightarrow G$ where $A\stackrel {G}{\simeq} \Phi(E)$ and
 $E/\Phi(E)\simeq G$.   Hence $E\in\mathfrak{F}$.

  Since $A$ is a  faithful  $\mathbb{F}_pG$-module and $G$ is a non-abelian simple group, we see that there is a chief factor $N/K$ of $E$ below $\Phi(E)$ with $E/C_E(N/K)\simeq G$.   Note that $H=(N/K)\rtimes (E/C_E(N/K))\in\mathfrak{F}$ by \cite[Corollary 2.2.5]{s9}. Since $N/K$ is the unique minimal normal subgroup of $H$, we see that
 $\tilde{\mathrm{F}}(H)\simeq N/K<H=H_\mathfrak{F}$, a contradiction.  Thus $\mathfrak{N}=\mathfrak{F}$.

 (2)
 From $\mathrm{F}^*(G)\subseteq G_\mathfrak{F}$ for every group  $G$ it follows that  $\mathfrak{N}^*\subseteq \mathfrak{F}$. Assume that $\mathfrak{F}\setminus\mathfrak{N}^*\neq\emptyset$. Let chose a minimal order group $G$ from
 $\mathfrak{F}\setminus\mathfrak{N}^*$. It is clear that $G=G_\mathfrak{F}=\tilde{\mathrm{F}}(G)$.  Since
 $\mathfrak{F}$ and $\mathfrak{N}^*$ are formations, we see that  $G$ has a unique minimal normal subgroup $N$. If $\Phi(G)=1$, then
 $G=\mathrm{Soc}(G)\in\mathfrak{N}^*$, a contradiction. Hence $N\leq\Phi(G)$ and  $N$ is a normal elementary abelian   $p$-subgroup of $G$ for some prime $p$.
 By our assumption $G/N\in\mathfrak{N}^*$. If $C_G(N)=G$, then $G$ is  quasinilpotent by the definition of a quasinilpotent group, a contradiction. Thus $C_G(N)\neq G$.
Now  $N$ is a unique minimal normal subgroup of
  $H=N\rtimes (G/C_G(N))$ and $H\in\mathfrak{F}$ by \cite[Corollary 2.2.5]{s9}. Note that $\Phi(H)=1$.
 Hence $\tilde{\mathrm{F}}(G)=N$ and $H_\mathfrak{F}=H$, the contradiction. Thus
 $\mathfrak{N}^*=\mathfrak{F}$.
\end{proof}

\begin{corollary}
If $\mathfrak{F}(\gamma)$ is a formation, then $\mathfrak{F}(\gamma)=\mathfrak{N}^*$.
\end{corollary}

Recall that the Frattini subgroup is the intersection of all maximal subgroups. The Jacobson radical is   the analog of the Frattini subgroup in the ring theory. It is the intersection of all maximal (right) ideals. Note that the concept of ideal in the ring theory corresponds to the concept of normal subgroup in the group theory. Here we proved:

\begin{theorem}\label{kolc}
   A group $G=\mathrm{\tilde F}(G)$  iff the intersection of all maximal normal subgroup of  $G$ coincides with $\Phi(G)$.
\end{theorem}

\begin{proof}
Let $M(G)$ be the intersection of all (proper) maximal normal subgroups of  $G$.

  $(a)$ \emph{$M(G)=\Phi(G)=1$ iff $G=\mathrm{Soc}(G)$.}

  Assume that $M(G)=\Phi(G)=1$ and $G$ has a not irreducible indecomposable  direct multiplier $K$. Let $N$ be a minimal normal subgroup of  $G$ and $N\leq K$. From $M(G)=\Phi(G)=1$ it follows that there is a maximal normal subgroup $M$ of $G$ with $N\not\leq M$. Note that $N\cap M=1$. From $M<MN\trianglelefteq G$ it follows that $MN=G$.  So $G=N\times M$. Thus $K=G\cap K=NM\cap K=N(M\cap K)=N\times (M\cap K),$ a contradiction.

Assume now that $G=\mathrm{Soc}(G)$. Note that $\Phi(G)=1$. Since every minimal normal subgroup is a direct product of simple groups, we see that $G$ is a direct product of simple groups. Not that if we take all multipliers in this product except one, then we obtain a maximal normal subgroup of $G$. The intersection of all such subgroups is 1.

$(b)$ $M(G/\Phi(G))=M(G)/\Phi(G)$.

Note that $M$ is a maximal normal subgroup of  $G$ if  $M/\Phi(G)$ is a maximal normal subgroup of $G/\Phi(G)$. Thus  $M(G/\Phi(G))=M(G)/\Phi(G)$.

Let prove the main statement of the theorem.
If $M(G)=\Phi(G)$, then $M(G/\Phi(G))=\Phi(G/\Phi(G))\simeq 1$ by $(b)$.  So $G/\Phi(G)=\mathrm{Soc}(G/\Phi(G))$ by $(a)$. Thus $G=\mathrm{\tilde F}(G)$.

If $\mathrm{\tilde F}(G)=G$, then $\mathrm{\tilde F}(G)/\Phi(G)=\mathrm{Soc}(G/\Phi(G))$. Hence $M(G/\Phi(G))=\Phi(G/\Phi(G))\simeq 1$ by $(a)$. Thus    $M(G)=\Phi(G)$ by $(b)$.\end{proof}

\section{A construction of $\mathbb{F}$-functorials}

Recall that the Fitting subgroup is the intersection of centralizers of all chief factors.
For a chief factor $H/K$ of  $G$ the subgroup $C_G^*(H/K)=HC_G(H/K)$ is called an \emph{inneriser} (see \cite[Definition 1.2.2]{s9}). It is the set of all elements of $G$ that induce inner automorphisms on $H/K$. From the definition of the generalized Fitting subgroup it follows that it is the intersection of  innerisers of all chief factors.
In this section we will obtain the similar characterization for some $\mathbb{F}$-functorials.

\begin{theorem}\label{thm2}
Let  $\varphi$ be a  Frattini functorial that satisfies: $H/K\not\leq\varphi(G/K)$ for every chief factor $H/K$ of $G$ with $\varphi(G)\leq K\leq H\leq\Phi(G)$. Then

   $$(\varphi\star\mathrm{F}^*)(G)=\bigcap_{H/K\not\leq\varphi(G/K) \textrm{ is a chief factor of }G}C^*_G(H/K).$$
\end{theorem}

\begin{corollary}\label{1.4} Let   $G$ be a group. Then

$$\mathrm{\tilde F}(G)=\bigcap_{H/K \textrm{ is a non-Frattini chief factor of }G}\hspace{-10mm}HC_G(H/K) \textrm{ and } $$
$$(\Phi_\pi\star\mathrm{F}^*)(G)=\bigcap_{H/K\not\leq\mathrm{O}_\pi(\Phi(G/K)) \textrm{ is a chief factor of }G}\hspace{-10mm}HC_G(H/K). $$
\end{corollary}

We need the following lemma in the proof of Theorem \ref{thm2}.

\begin{lemma}\label{Lem1}
  Let $H/K$ and $M/N$ be  $G$-isomorphic chief factors of a group $G$. Then $C_G^*(H/K)=C_G^*(M/N)$.
\end{lemma}

\begin{proof}
  Since $H/K$ and $M/N$ are  $G$-isomorphic chief factors of $G$, we see that $C_G(H/K)=C_G(M/N)=C$. If $H/K$ is abelian, then  $C=C_G^*(H/K)=C_G^*(M/N)$. Assume now that $H/K$ is non-abelian. Note that $K\leq C$ and $N\leq C$. Since $H/K$ and $M/N$ are non-abelian chief factors of $G$, we see that $H\cap C=K$ and $M\cap C=K$. Now $$C^*_G(H/K)/C=HC/C\simeq H/K\simeq M/N\simeq MC/C=C^*_G(M/N)/C.$$
  Note that $G/C$ has a unique minimal normal subgroup $L/C$ and $L/C\simeq H/K$. It means that $C^*_G(H/K)/C=L/C=C^*_G(M/N)/C$. Thus $C^*_G(H/K)=C^*_G(M/N)$.
\end{proof}

\begin{proof}[Proof of Theorem \ref{thm2}]
Denote
$$D=\bigcap_{H/K\not\leq\varphi(G/K) \textrm{ is a chief factor of }G}C^*_G(H/K).$$

  Let   $H/K\not\leq\varphi(G/K)$ be a chief factor of  $G$. If $(\varphi\star\mathrm{F}^*)(G)K\cap H=K$, then   $H/K$ and $(\varphi\star\mathrm{F}^*)(G)K/K$ permute elementwise.
   So $$(\varphi\star\mathrm{F}^*)(G)\leq C_G(H/K)\leq HC_G(H/K)=C^*_G(H/K).$$
   Suppose now $(\varphi\star\mathrm{F}^*)(G)K\cap H\neq K$. Then $H/K\leq(\varphi\star\mathrm{F}^*)(G)K/K$.

  Assume that $H\cap K\varphi(G)=H$. Then $H\varphi(G)= K\varphi(G)$ and
  $$H/K\subseteq H\varphi(G)/K=\varphi(G)K/K\subseteq \varphi(G/K),$$
   a contradiction.

  Hence chief factors $H\varphi(G)/(K\varphi(G))$ and $H/K$ are $G$-isomorphic. By Lemma \ref{Lem1} we have $C_G^*(H/K)=C_G^*(H\varphi(G)/(K\varphi(G)))$. Now  $(\varphi\star\mathrm{F}^*)(G)K/(K\varphi(G))$ is a quasinilpotent group and $H\varphi(G)/(K\varphi(G))\leq(\varphi\star\mathrm{F}^*)(G)K/(K\varphi(G))$.
  Since $H\varphi(G)/(K\varphi(G))$ is a chief factor of $G$ and  $(\varphi\star\mathrm{F}^*)(G)K\trianglelefteq G$, we see that $H\varphi(G)/(K\varphi(G))$ is the product of minimal normal subgroups of a quasinilpotent group $(\varphi\star\mathrm{F}^*)(G)K/(K\varphi(G))$. Now every element of $K(\varphi\star\mathrm{F}^*)(G)$ induces an inner automorphism on $H\varphi(G)/(K\varphi(G))$. So  \begin{align*}(\varphi\star\mathrm{F}^*)(G)\leq (\varphi\star\mathrm{F}^*)(G)K&=H\varphi(G)C_{(\varphi\star\mathrm{F}^*)(G)K}(H\varphi(G)/(K\varphi(G)))\\
  &\leq H\varphi(G)C_G(H\varphi(G)/(K\varphi(G))).\end{align*} Hence $(\varphi\star\mathrm{F}^*)(G)\leq D$.

 Let $$ D_1=\bigcap_{H/K \textrm{ is a non-Frattini chief factor of }G}\hspace{-10mm}HC_G(H/K).$$  It is clear that $\Phi(G)\subseteq D_1$ and $D\subseteq D_1$. Note that every element of   $D_1/\Phi(G)$ induces an inner automorphism on every  $G$-composition factor of $\mathrm{Soc}(G/\Phi(G))$ and $\mathrm{Soc}(G/\Phi(G))$ is the product of minimal normal subgroups of $G/\Phi(G)$. Thus every element $x\Phi(G)\in D_1/\Phi(G)$ induces an inner automorphism on $\mathrm{Soc}(G/\Phi(G))$. So there exists $y\Phi(G)\in\mathrm{Soc}(G/\Phi(G))$ such that $xy^{-1}\Phi(G)$ acts trivially on    $\mathrm{Soc}(G/\Phi(G))$. From $C_{G/\Phi(G)}(\mathrm{\tilde F}(G/\Phi(G)))\subseteq \mathrm{\tilde F}(G/\Phi(G))=\mathrm{Soc}(G/\Phi(G))$ it follows that $xy^{-1}\Phi(G)\in \mathrm{Soc}(G/\Phi(G))$. Hence $D_1/\Phi(G)\subseteq \mathrm{Soc}(G/\Phi(G))=\mathrm{\tilde F}(G/\Phi(G))$. Thus $D\subseteq D_1\subseteq\mathrm{\tilde F}(G)$.

Now $G$ has the following normal series
$$ 1\trianglelefteq \varphi(G)\trianglelefteq \Phi(G)\trianglelefteq D\trianglelefteq\mathrm{\tilde F}(G)\trianglelefteq G.$$
 Since $\mathrm{\tilde F}(G)/\Phi(G)=\mathrm{Soc}(G/\Phi(G))$ and $\Phi(G/\Phi(G))\simeq 1$, we see that $H/K\not\leq \varphi(G/K)$ for every chief factor of $G$ between $\Phi(G)$ and $D$. Note that the same holds for every chief factor of $G$ between $\varphi(G)$  and $\Phi(G)$. Consider the chief series of $G$ between $\varphi(G)$ and $D$ such that it contains $\Phi(G)$.      Let $N/M$ be a chief factor of  it. Then  $D=NC_G(N/M)\cap D=N(D\cap C_G(N/M))= NC_D(N/M)$. So every element of $D$ induces an inner automorphism on    $N/M$. Hence $D/\varphi(G)$ has a normal series such that every element of $D/\varphi(G)$   induces an inner automorphism on every its factor.  Thus $D/\varphi(G)$ is quasinilpotent by \cite[X, Lemma  13.1]{19}, i.e.   $D\leq (\varphi\star\mathrm{F}^*)(G)$. Thus $(\varphi\star\mathrm{F}^*)(G)=D$.
\end{proof}

\begin{proof}[Proof of Corollary \ref{1.4}]
Recall \cite{Foerster1985} that $ \Phi\star\mathrm{F}^*=\mathrm{\tilde F}$. Now the first part of Corollary \ref{1.4} directly follows from Theorem \ref{thm2}. Note that $ \mathrm{O}_\pi$ satisfies $(F1)$ and $(F2)$. Now it is clear that $ \mathrm{O}_\pi(\Phi(G))$ is a Frattini functorial. Since $\Phi(G)$ is nilpotent, we see that $\Phi(G)/ \mathrm{O}_\pi(\Phi(G))$ is a $ \pi'$-group. Hence $H/K\not\leq \mathrm{O}_\pi(\Phi(G/K))\simeq 1$ for every chief factor $H/K$ of $G$ with $\mathrm{O}_\pi(\Phi(G))\leq K<H\leq\Phi(G)$. Thus the second part of Corollary \ref{1.4} directly follows from Theorem \ref{thm2}.
\end{proof}

\section{The generalizations of the Fitting height}

One of the important invariants of a soluble group $G$ is its Fitting height $h(G)$. This invariant is defined with the help of the Fitting subgroup.  E.I.~Khukhro and P.~Shumyatsky \cite{KHUKHRO2015, Khukhro2017} defined the generalized Fitting height $h^*(G)$ of a group $G$ with the help of the generalized Fitting subgroup. They studied this height of factorized groups.  Here we introduce the height $h_\gamma(G)$ of a group $G$ which corresponds to a given $\mathbb{F}$-functorial $\gamma$.

\begin{definition}
  Let $\gamma$ be a functorial with $\gamma(G)>1$ for every non-unit group $G$. Then the $\gamma$-series of $G$ is defined  starting from $\gamma_{(0)}(G)=1$, and then by induction $\gamma_{(i+1)}(G)=(\gamma_{(i)}\star\gamma)(G)$ is the inverse image of $\gamma(G/\gamma_{(i)}(G))$. The least number $h$ such that $\gamma_{(h)}(G)=G$ is defined to be $\gamma$-height $h_\gamma(G)$ of $G$.
\end{definition}

Since $1\neq \mathrm{F}^*(G)\subseteq \gamma(G)$ for any $G\neq 1$ and  $\mathbb{F}$-functorial $\gamma$, we see that the $\gamma$-height is defined for any $ \mathbb{F}$-functorial $\gamma$.
   Note that if $G$ is a soluble group and $\gamma$ is an $\mathbb{F}$-functorial, then $h_\gamma(G)=h(G)$ is the Fitting height of $G$.
If $\gamma=\mathrm{F}^*$, then $h_\gamma(G)=h^*(G)$.

\begin{theorem}\label{len1}
  Let $\gamma$ be an $\mathbb{F}$-functorial. Then $h_{\mathrm{\tilde{F}}}(G)\leq h_\gamma(G)\leq 2h_{\mathrm{\tilde{F}}}(G)$ for any group $G$.  For any natural $n$ there exists a group $H$ with $h_{\mathrm{\tilde{F}}}(H)=n$ and $h^*(H)=2n$.
\end{theorem}

\begin{proof}
  Let $\gamma$ be an $\mathbb{F}$-functorial.  Since $\Phi(G)$ and $\mathrm{Soc}(G/\Phi(G))$ are quasinilpotent, we see that $\gamma(G)\leq\mathrm{\tilde F}(G)\leq \mathrm{F}_{(2)}^*(G)\leq \gamma_{(2)}(G)$.  Now $\gamma_{(n)}(G)\leq\mathrm{\tilde F}_{(n)}(G)\leq \gamma_{(2n)}(G)$. Hence if $\mathrm{\tilde F}_{(n)}(G)=G$, then $\gamma_{(n)}(G)\leq G$ and $\gamma_{(2n)}(G)=G$. It means $h_{\mathrm{\tilde{F}}}(G)\leq h_\gamma(G)\leq 2h_{\mathrm{\tilde{F}}}(G)$.

Let $K$ be a group, $K_1$ be isomorphic to the regular wreath product of  $\mathbb{A}_5$ and $K$. Note that the base $B$  of it is the unique minimal normal subgroup of $K_1$ and non-abelian. According to \cite{Griess1978}, there is a
 Frattini $\mathbb{F}_3K_1$-module $A$ which is faithful for $K_1$ and a Frattini extension  $A\rightarrowtail K_2\twoheadrightarrow K_1$
such that $A\stackrel {K_1}{\simeq} \Phi(K_2)$ and $K_2/\Phi(K_2)\simeq K_1$.

Let denote $K_2$ by $\mathbf{f}(K)$.
Now $\mathbf{f}(K)/\mathrm{\tilde F}(\mathbf{f}(K))\simeq K$. From the definition of $h_{\mathrm{\tilde F}}$ it follows that $h_{\mathrm{\tilde F}}(\mathbf{f}(K))=h_{\mathrm{\tilde F}}(K)+1$.

Note that $\Phi(\mathbf{f}(K))\subseteq \mathrm{F}^*(\mathbf{f}(K))$. Assume that $\Phi(\mathbf{f}(K))\neq \mathrm{F}^*(\mathbf{f}(K))$. It means that $\mathrm{F}^*(\mathbf{f}(K))=\mathrm{\tilde F}(\mathbf{f}(K))$ is quasinilpotent. By \cite[X, Theorem 13.8]{19} it follows that $\Phi(\mathbf{f}(K))\subseteq \mathrm{Z}(\mathrm{F}^*(\mathbf{f}(K)))$. It means that $1<B\leq C_{K_1}(A)$. Thus $A$ is not faithful, a contradiction.

Thus  $\Phi(\mathbf{f}(K))= \mathrm{F}^*(\mathbf{f}(K))$ and $\mathbf{f}(K)/\mathrm{F}^*(\mathbf{f}(K))\simeq K_1$. Since $K_1$ has a unique minimal normal subgroup $B$ and it is non-abelian, we see that $\mathrm{F}^*(K_1)=B$. It means that $\mathbf{f}(K)/\mathrm{F}^*_{(2)}(\mathbf{f}(K))\simeq K$. From the definition of $h^*$ it follows that $h^*(\mathbf{f}(K))=h^*(K)+2$.

As usual, let $\mathbf{f}^{(1)}(K)=\mathbf{f}(K)$ and $\mathbf{f}^{(i+1)}(K)=\mathbf{f}(\mathbf{f}^{(i)}(K))$. Then   $h_{\mathrm{\tilde{F}}}(\mathbf{f}^{(n)}(1))=n$ and $h^*(\mathbf{f}^{(n)}(1))=2n$ for any natural $n$.
\end{proof}

Recall \cite[Definition 4.1.1]{PFG} that a group $G$ is called a mutually permutable product of its subgroups $A$ and $B$ if $G=AB$,  $A$ permutes with every subgroup of $B$ and    $B$ permutes with every subgroup of $A$. The products of mutually permutable subgroups is the very interesting topic of the theory of groups (see \cite[Chapter 4]{PFG}). The main result of this section is

\begin{theorem}\label{cor81}
  Let a group $G$ be the product of the mutually permutable subgroups $A$ and $B$. Then $\max\{h^*(A), h^*(B)\}\leq h^*(G)\leq \max\{h^*(A), h^*(B)\}+1$.
\end{theorem}

\begin{corollary}\label{cor82}
  Let a soluble group $G$ be the product of the mutually permutable subgroups $A$ and $B$. Then $\max\{h(A), h(B)\}\leq h(G)\leq \max\{h(A), h(B)\}+1$.
\end{corollary}

\begin{example}
  Note that the symmetric group $\mathbb{S}_3$ of degree $3$ is the mutually permutable product of the cyclic groups $Z_2$ and $Z_3$  of orders $2$ and $ 3$ respectively.  Hence $h^*(\mathbb{S}_3)=\max\{h^*(Z_2), h^*(Z_3)\}+1=\max\{h(Z_2), h(Z_3)\}+1$.
\end{example}

The proof of Theorem \ref{cor81} is rather complicated and require several steps. Some of them hold not only for $h^*$. At first we will study the products of normal subgroups.

\begin{theorem}\label{len3}
Let $\gamma$ be  a functorial with $\gamma(G)>1$ that satisfies $(F1)$ and $(F2)$.

$(1)$  If $G=\times_{i=1}^n  A_i$ is the direct product of its normal subgroups $A_i$, then $h_\gamma(G)= \max\{h_\gamma(A_i)\mid 1\leq i\leq n\}$.

$(2)$ Let $G=\langle A_i\mid 1\leq i\leq n\rangle$ be the join of its subnormal subgroups $A_i$. Then $h_\gamma(G)\leq \max\{h_\gamma(A_i)\mid 1\leq i\leq n\}$. If $ \gamma$ satisfies $(F5)$, then $h_\gamma(G)= \max\{h_\gamma(A_i)\mid 1\leq i\leq n\}$.
\end{theorem}

\begin{proof}
Note that $\gamma_{(n)}$ satisfies $(F1)$ and $(F2)$ for every $n$ by Proposition \ref{length0}.

$(1)$ From $(1)$ of Proposition \ref{FF} it follows that if   $G=\times_{i=1}^n  A_i$, then $\gamma_{(n)}(G)=\times_{i=1}^n  \gamma_{(n)}(A_i)$. It means that $h_\gamma(G)= \max\{h_\gamma(A_i)\mid 1\leq i\leq n\}$.

$(2)$ Assume  that $G=\langle A_i\mid 1\leq i\leq n\rangle$ is the join of its subnormal subgroups $A_i$,  $h_1=\max\{h_\gamma(A_i)\mid 1\leq i\leq n\}$ and $h_2=h_\gamma(G)$. Since $\gamma_{(n)}$ satisfies $(F2)$, we see that  $\gamma_{(n)}(N)\subseteq \gamma_{(n)}(G)$ for every subnormal subgroup $N$ of $G$ and every $n$. Now $$G=\langle A_i\mid 1\leq i\leq n\rangle=\langle \gamma_{(h_1)}(A_i)\mid 1\leq i\leq n\rangle\subseteq \gamma_{(h_1)}(G)\subseteq G.$$ Hence $\gamma_{(h_1)}(G)=G$. It means that $h_2\leq h_1$.

  Suppose that $\gamma$ satisfies $(F5)$. Now  $\gamma_{(n)}$ satisfies $(F5)$
  for every $n$ by Proposition \ref{length0}. Since $\gamma_{(n)}$ satisfies $(F2)$ and $(F5)$, we see that $\gamma_{(n)}(G)\cap N=\gamma_{(n)}(N)$ for every subnormal subgroup $N$ of $G$. Now   $A_i=A_i\cap G=A_i\cap \gamma_{(h_2)}(G)=\gamma_{(h_2)}(A_i)$. It means that $h_\gamma(A_i)\leq h_2$ for every $i$. Hence $h_1\leq h_2$. Thus $h_1=h_2$.
\end{proof}

\begin{corollary}
  Let $G=\langle A_i\mid 1\leq i\leq n\rangle$ be the join of its subnormal subgroups $A_i$. Then $h^*(G)= \max\{h^*(A_i)\mid 1\leq i\leq n\}$.
\end{corollary}

\begin{example}
Let $E\simeq \mathbb{A}_5$. According to $(2)$ of the proof of Theorem \ref{lattice} there is an $\mathbb{F}_5E$-module $V$ such
 that $R=Rad(V)$ is a faithful irreducible $\mathbb{F}_5E$-module and $V/R$ is
 an irreducible trivial $\mathbb{F}_5E$-module.  Let $G=V\leftthreetimes E$. Now $\Phi(G)=R$ by \cite[B, Lemma 3.14]{s8}. Note that $G/\Phi(G)=G/R\simeq Z_5\times E$. So $\mathrm{\tilde{F}}(G)=\mathrm{\tilde{F}}^\infty(G)=G$ and  $h_{\mathrm{\tilde{F}}}(G)=h_{\mathrm{\tilde{F}}^\infty}(G)=1$. Note that $G=V(RE)$ where $V$ and $RE$ are normal subgroups of $G$. Since $ V$ is abelian, we see that $h_{\mathrm{\tilde{F}}}(V)=h_{\mathrm{\tilde{F}}^\infty}(V)=1$. Note that $R$ is a unique minimal normal subgroup of $RE$ and $\Phi(RE)=1$. It means that  $\mathrm{\tilde{F}}(RE)=\mathrm{\tilde{F}}^\infty(RE)=R$ and $h_{\mathrm{\tilde{F}}}(RE)=h_{\mathrm{\tilde{F}}^\infty}(RE)=2$. Thus $h_{\mathrm{\tilde{F}}}(G)<\max\{h_{\mathrm{\tilde{F}}}(V), h_{\mathrm{\tilde{F}}}(RE)\}$ and  $h_{\mathrm{\tilde{F}}^\infty}(G)<\max\{h_{\mathrm{\tilde{F}}^\infty}(V), h_{\mathrm{\tilde{F}}^\infty}(RE)\}$.
\end{example}

Let $\mathfrak{F}$ be a formation. Recall that the $ \mathfrak{F}$-\emph{residual} of a group $ G$ is the smallest normal subgroup $G^\mathfrak{F}$ of $ G$ with $ G/G^\mathfrak{F}\in\mathfrak{F}$.

\begin{lemma}\label{minusone}
  If $H\neq 1$, then $h^*(H^{\mathfrak{N}^*})=h^*(H)-1$.
\end{lemma}

\begin{proof}
  Let prove that if $H\neq 1$, then $h^*(H^{\mathfrak{N}^*})=h^*(H)-1$. Let $h^*(H)=n$ and $h^*(H^{\mathfrak{N}^*})=k$. Then $\mathrm{F}_{(n-1)}^*(H)<H$ and $H/\mathrm{F}_{(n-1)}^*(H)$ is quasinilpotent. It means that $H^{\mathfrak{N}^*}\leq \mathrm{F}_{(n-1)}^*(H)$. Since $\mathrm{F}_{(n-1)}^*$ satisfies $(F5)$, we see that $\mathrm{F}_{(n-1)}^*(H^{\mathfrak{N}^*})=H^{\mathfrak{N}^*}$. Hence  $k\leq n-1$.

 Note that $H=\mathrm{F}_{(k)}^*(H^{\mathfrak{N}^*})\leq \mathrm{F}_{(k)}^*(H)$. It means that $H/\mathrm{F}_{(k)}^*(H)$ is quasinilpotent. Hence $k\geq n-1$. Thus $k=n-1$.
\end{proof}

The next step in the proof of Theorem \ref{cor81} is to study mutually permutable products of quasinilpotent subgroups. We will need the following definitions and results in the proof.

A chief factor $H/K$ of  $G$ is called   $\mathfrak{X}$-\emph{central} in $G$ provided    $(H/K)\rtimes (G/C_G(H/K))\in\mathfrak{X}$ (see \cite[p. 127--128]{s6} or \cite[1, Definition 2.2]{Guo2015}). A normal subgroup $N$ of $G$ is said to be $\mathfrak{X}$-\emph{hypercentral} in $G$ if $N=1$ or $N\neq 1$ and every chief factor of $G$ below $N$ is $\mathfrak{X}$-central. The symbol $\mathrm{Z}_\mathfrak{X}(G)$ denotes the $\mathfrak{X}$-\emph{hypercenter} of $G$, that is, the product of all normal $\mathfrak{X}$-\emph{hypercentral} in $G$ subgroups. According to \cite[Lemma 14.1]{s6} or \cite[1, Theorem 2.6]{Guo2015} $\mathrm{Z}_\mathfrak{X}(G)$ is the largest normal $\mathfrak{X}$-hypercentral subgroup of $G$. If $\mathfrak{X}=\mathfrak{N}$ is the class of all nilpotent groups, then $\mathrm{Z}_\mathfrak{N}(G)=\mathrm{Z}_\infty(G)$ is the hypercenter of $G$.

If $\mathfrak{F},\mathfrak{H},\mathfrak{K}\neq\emptyset$ are formations, then $\mathfrak{FH}=(G\mid G^\mathfrak{H}\in\mathfrak{F})$ is a formation, $G^{\mathfrak{FH}}=(G^\mathfrak{H})^\mathfrak{F}$ and $(\mathfrak{FH})\mathfrak{K}=\mathfrak{F}(\mathfrak{HK})$  \cite[IV, Theorem 1.8]{s8}. That is why the class $(\mathfrak{N}^*)^n=\underbrace{\mathfrak{N}^*\dots\mathfrak{N}^*}_n$ is a well defined formation.

\begin{lemma}\label{form}
  Let $n$ be a natural number.  Then   $(\mathfrak{N}^*)^n=(G\mid h^*(G)\leq n)=(G\mid G=\mathrm{Z}_{(\mathfrak{N}^*)^n}(G))$.
\end{lemma}

\begin{proof}
 It is well known that the class of all quasinilpotent groups is a composition (or Baer-local, or solubly saturated) formation (see \cite[Example 2.2.17]{s9}).  According to  \cite[Theorem 7.9]{s6} $(\mathfrak{N}^*)^n $ is a composition  formation. Now $(\mathfrak{N}^*)^n=(G\mid G=\mathrm{Z}_{(\mathfrak{N}^*)^n}(G))$ by \cite[1, Theorem 2.6]{Guo2015}.

From Lemma \ref{minusone} it follows that if $G\in(G\mid h^*(G)\leq n)$, then $G^{(\mathfrak{N}^*)^n}=1$. It means that $(G\mid h^*(G)\leq n)\subseteq (\mathfrak{N}^*)^n$. Assume that there is a group $G\in(\mathfrak{N}^*)^n$ with $h^*(G)>n$. Then $h^*(G^{(\mathfrak{N}^*)^n})>0$ by Lemma \ref{minusone}. It means that $G^{(\mathfrak{N}^*)^n}\neq 1$, a contradiction. Therefore $(\mathfrak{N}^*)^n\subseteq (G\mid h^*(G)\leq n)$. Thus $(\mathfrak{N}^*)^n=(G\mid h^*(G)\leq n)$.
\end{proof}

\begin{lemma}\label{metaquasi}
   If a group $G=AB$ is a product of mutually permutable quasinilpotent subgroups $A$ and $B$, then $h^*(G)\leq 2$.\end{lemma}

\begin{proof}
To prove this lemma we need only to prove that if a group $G=AB$ is a product of mutually permutable quasinilpotent subgroups $A$ and $B$, then $G\in(\mathfrak{N}^*)^2$ by Lemma \ref{form}.
 Assume the contrary. Let $G$ be a minimal order counterexample.

$(1)$ \emph{$G$ has a unique minimal normal subgroup $N$ and    $G/N\in (\mathfrak{N}^*)^2$.}

Note that $G/N$ is a mutually permutable product of quasinilpotent subgroups $(AN/N)$ and $(BN/N)$ by \cite[Lemma 4.1.10]{PFG}. Hence $G/N\in (\mathfrak{N}^*)^2$ by our assumption. Since $(\mathfrak{N}^*)^2$ is a formation, we see that $G$ has a unique minimal normal subgroup. 
According to \cite[Theorem 4.3.11]{PFG} $A_GB_G\neq 1$.  WLOG we may assume that $G$ has a minimal normal subgroup  $N\leq A$.

$(2)$ $N\leq A\cap B$.

Suppose that $N\cap B=1$. Then $A\leq C_G(N)$ or $B\leq C_G(N)$ by \cite[Lemma 4.3.3(5)]{PFG}.
If $A\leq C_G(N)$, then $N\rtimes G/C_G(N)\simeq N\rtimes B/C_B(N)\in(\mathfrak{N}^*)^2$.
If $B\leq C_G(N)$, then $N\rtimes G/C_G(N)\simeq N\rtimes A/C_A(N)\in(\mathfrak{N}^*)\subseteq(\mathfrak{N}^*)^2$ by \cite[Corollary 2.2.5]{s9}. In both cases $N\leq \mathrm{Z}_{(\mathfrak{N}^*)^2}(G)$. It means that $G\in(\mathfrak{N}^*)^2$, a contradiction.
Now $N\cap B\neq 1$. Hence $N\leq A\cap B$ by \cite[Lemma 4.3.3(4)]{PFG}.

$(3)$ \emph{$N$ is non-abelian.}

Assume that $N$ is abelian. Since $A$ is quasinilpotent, we see that   $A/C_A(N)$ is a $p$-group. By analogy $B/C_B(N)$ is a $p$-group. Note that $A/C_A(N)\simeq AC_G(N)/C_G(N)$ and $B/C_B(N)\simeq BC_G(N)/C_G(N)$. From $G=AB$ it follows that $G/C_G(N)$ is a $p$-group. Since $N$ is a chief factor of $G$, we see that       $G/C_G(N)\simeq 1$. So $N\leq\mathrm{Z}_\infty(G)\leq \mathrm{Z}_{(\mathfrak{N}^*)^2}(G)$. Thus $G\in(\mathfrak{N}^*)^2$, a contradiction.
It means that $N$ is non-abelian.

$(4)$ \emph{The final contradiction.}

Now $N$ is a direct product of minimal normal subgroups of $A$. Since $A$ is quasinilpotent, we see that every element of $A$ induces an inner automorphism on every  minimal normal subgroup of $A$. Hence     every element of $A$ induces an inner automorphism on $N$. By analogy  every element of $B$ induces an inner automorphism on $N$. From $G=AB$ it follows that  every element of $G$ induces an inner automorphism on $N$. So $NC_G(N)=G$ or $G/C_G(N)\simeq N$. Now $N\rtimes(G/C_G(N))\in(\mathfrak{N}^*)^2$. It means that $N\leq \mathrm{Z}_{(\mathfrak{N}^*)^2}(G)$. Thus $G\in(\mathfrak{N}^*)^2$ and $h^*(G)\leq 2$, the final contradiction.
\end{proof}

Let $\gamma$ be an $\mathbb{F}$-functorial.  Now we are ready to give some bounds on the $ \gamma$-height of a mutually permutable product.

\begin{theorem}\label{len2}
  Assume that $\gamma$ satisfies $(F1)$, $(F2)$ and $(F3)$. If a group $G=AB$ is the mutually permutable product of its subgroups $A$ and $B$, then  $h_\gamma(G)\leq \max\{h_\gamma(A^{\mathfrak{N^*}}),h_\gamma(B^{\mathfrak{N^*}})\}+2$.
Moreover, if $\gamma$ satisfies $(F5)$, then $\max\{h_\gamma(A^{\mathfrak{N}^*}), h_\gamma(B^{\mathfrak{N}^*})\}\leq h_\gamma(G)$.
\end{theorem}

\begin{proof}
From $(2)$ of Proposition \ref{FF} it follows that $\mathrm{F}^*(G)\subseteq\gamma(G)$ for every group $G$. It means that $h_\gamma(G)\leq h^*(G)$  holds for every groups $G$.

Let a group $G=AB$ be the product of mutually permutable subgroups $A$ and $B$. Note that $A'$ and $B'$ are subnormal in $G$ by \cite[Corollary 4.1.26]{PFG}.  Since $H^{\mathfrak{N}^*}\trianglelefteq H'$ holds for every group $H$, subgroups $A^{\mathfrak{N}^*}$ and $B^{\mathfrak{N}^*}$ are subnormal in $G$. Let $C=\langle A^{\mathfrak{N}^*}, B^{\mathfrak{N}^*}\rangle^G=\langle\{(A^{\mathfrak{N}^*})^x\mid x\in G\}\cup\{(B^{\mathfrak{N}^*})^x\mid x\in G\}\rangle$. Then by $(2)$ of Theorem \ref{len3}
$$h_\gamma(C)\leq\max\left\{\{(h_\gamma(A^{\mathfrak{N}^*})^x)\mid x\in G\}\cup\{(h_\gamma(B^{\mathfrak{N}^*})^x)\mid x\in G\}\right\}=\max\{h_\gamma(A^{\mathfrak{N}^*}), h_\gamma(B^{\mathfrak{N}^*})\}.$$
Now $G/C=(AC/C)(BC/C)$ is a mutually permutable product of quasinilpotent subgroups $AC/C$ and $BC/C$ by
\cite[Lemma 4.1.10]{PFG}. It means that $h_\gamma(G/C)\leq h^*(G/C)\leq 2$ by Lemma \ref{metaquasi}. Thus  $h_\gamma(G)\leq h_\gamma(C)+h_\gamma(G/C)\leq  \max\{h_\gamma(A^{\mathfrak{N}^*}), h_\gamma(B^{\mathfrak{N}^*})\}+2$.

Assume that $ \gamma$ satisfies $(F5)$. Now $\gamma_{(n)}$ satisfies $(F5)$ for all $n$ by Proposition \ref{length0}. It means that   $h_\gamma(G)\geq h_\gamma(C)=\max\{h_\gamma(A^{\mathfrak{N}^*}), h_\gamma(B^{\mathfrak{N}^*})\}$.
\end{proof}

\begin{lemma}\label{lem4}
  Let $N$ be a normal subgroup of a group $G$. If $N$ is a direct product of isomorphic simple groups and $h^*(G/C_G^*(N))\leq k-1$, then $\mathrm{F}^*_{(k)}(G/N)=\mathrm{F}^*_{(k)}(G)/N$.
\end{lemma}

\begin{proof}
Assume that $h^*(G/C_G^*(N))\leq k-1$.
Let $F/N=\mathrm{F}^*_{(k)}(G/N)$. Then $\mathrm{F}^*_{(k)}(G)\subseteq F$. Now $F/C_F^*(N)\simeq FC_G^*(N)/C_G^*(N)\trianglelefteq G/C_G^*(N)$.  Hence $h^*(F/C_F^*(N))\leq k-1$. It means that $h^*(F/C_F^*(H/K))\leq k-1$ for every chief factor $H/K$ of $F$ below $N$. Hence $(H/K)\rtimes (F/C_F(H/K))\in (\mathfrak{N}^*)^k$  for every chief factor $H/K$ of $F$ below $N$. It means that $N\leq \mathrm{Z}_{(\mathfrak{N}^*)^k}(F)$. Thus $F\in (\mathfrak{N}^*)^k$ by Lemma \ref{form}. So $F\subseteq \mathrm{F}^*_{(k)}(G)$. Thus $\mathrm{F}^*_{(k)}(G)= F$.
\end{proof}

\begin{proof}[Proof of Theorem \ref{cor81}]
   If $A=1$ or $B=1$, then there is nothing to prove. Assume now that $A, B\neq 1$. Now $h^*(A^{\mathfrak{N}^*})=h^*(A)-1$ and $h^*(B^{\mathfrak{N}^*})=h^*(B)-1$ by Lemma \ref{minusone}. Note that $\max\{h^*(A^{\mathfrak{N^*}}),h^*(B^{\mathfrak{N^*}})\}=\max\{h^*(A), h^*(B)\}-1$. From Theorem \ref{len2} it follows that
     $\max\{h^*(A), h^*(B)\}-1\leq h^*(G)\leq \max\{h^*(A), h^*(B)\}+1$.

Assume that a group $G=AB$ is a minimal order mutually permutable product of groups  $A$ and $B$ with $h^*(G)<\max\{h^*(A), h^*(B)\}$. WLOG we may assume that $h^*(A)>h^*(G)$ and $h^*(A)\geq h^*(B)$. Since  $\max\{h^*(A), h^*(B)\}-1\leq h^*(G)$, we see that $h^*(A)=h^*(G)+1$.

Let $N$ be a minimal normal subgroup of $G$. Then $ N\cap A\in\{N, 1\}$ by \cite[Lemma 4.3.3(4)]{PFG}.

Assume that $N\cap A=1$. Now $ G/N=(AN/N)(BN/N)$ is a mutually permutable product of groups  $AN/N$ and $BN/N$ by \cite[Lemma 4.1.10]{PFG}. By our assumption  and $h^*(G)\geq h^*(G/N)\geq h^*(AN/N)=h^*(A)$, a contradiction. Hence $N\cap A=N$ for every minimal normal subgroup $N$ of $G$. Note that $N\leq \mathrm{F}^*(A)$.

 Now  $h^*(G)+1=h^*(A)>h^*(G)\geq h^*(G/N)\geq h^*(A/N)\geq h^*(A)-1$. It means that $h^*(G)=h^*(A/N)= h^*(A)-1$. If $G$ has two  minimal normal subgroups $N_1$ and $N_2$, then $h^*(A/N_1)=h^*(A/N_2)= h^*(A)-1$. It means $h^*(A)<h^*(A)-1$ by Lemma \ref{form}, a contradiction. Hence $G$ has a unique minimal normal subgroup $N$.

 Assume that $h^*(A/C_A^*(N))<h^*(A)-1$. Then $\mathrm{F}_{(h^*(A)-1)}^*(A/N)=\mathrm{F}_{(h^*(A)-1)}^*(A)/N<A/N$ by Lemma \ref{lem4}. It means that $h^*(A)=h^*(A/N)$. Since $G/N=(A/N)(BN/N)$ is a mutually permutable products of groups $A/N$ and $BN/N$ by \cite[Lemma 4.1.10]{PFG}, we see that $h^*(G)\geq h^*(G/N)\geq h^*(A/N)=h^*(A)>h^*(G)$ by our assumption, a contradiction. Hence  $h^*(A/C_A^*(N))=h^*(A)-1$

    Since $G/C_G^*(N)=(AC_G^*(N)/C_G^*(N))(BC_G^*(N)/C_G^*(N))$ is a mutually permutable products of subgroups $AC_G^*(N)/C_G^*(N)$ and $BC_G^*(N)/C_G^*(N)$ by \cite[Lemma 4.1.10]{PFG} and $A/C_A^*(N)\simeq AC_G^*(N)/C_A^*(N)$, we see that $h^*(G/C_G^*(N))\geq h^*(A/C_A^*(N))=h^*(A)-1$ by our assumptions. Note that $\mathrm{F}^*(G)\leq C_G^*(N)$. Now $h^*(G)-1=h^*(G/\mathrm{F}^*(G))\geq h^*(G/C_G^*(N))\geq h^*(A/C_A^*(N))=h^*(A)-1$. It means that $h^*(G)\geq h^*(A)$, the final contradiction.
      \end{proof}

\section{Some applications of $\mathbb{F}$-functorials}

Recall \cite{mv1} that a subgroup   $H$ of $G$  is called $R$-\emph{subnormal} if $H$ is subnormal in $\langle H, R\rangle$.
 In \cite{MonKon, MonChir, mv2, mv1} the products of $R$-subnormal subgroups were studied for   $R\in\{\mathrm{F}(G), \mathrm{F}^*(G)\}$. It was shown that if  $G$ is the product of two nilpotent (resp. quasinilpotent) $\mathrm{F}(G)$-subnormal (resp. $\mathrm{F}^*(G)$-subnormal) subgroups, then it is nilpotent (resp. quasinilpotent).

 \begin{theorem} The following holds:

 $(1)$ A group $G$ is nilpotent if and only if every maximal subgroup of $G$ is
 $\tilde{\mathrm{F}}(G)$-subnormal.

 $(2)$ Assume that $\mathfrak{F}$ is a Q-closed class of groups which contains every group whose maximal subgroups are  $\mathrm{F}^*(G)$-subnormal. Then $\mathfrak{F}$ is the class of all groups. \end{theorem}

 \begin{proof}  Let $G$ be a nilpotent group. It is clear that every maximal subgroup of $G$
 is $\tilde{\mathrm{F}}(G)$-subnormal.

  Conversely. Assume the theorem is false and let a group $G$ the the minimal order counterexample.

 Assume that $\Phi(G)\neq 1$.  Since $ \tilde {\mathrm{F}} (G / \Phi (G)) = \tilde {\mathrm{F}} (G) / \Phi (G) $, we
 see that every maximal subgroup of $ G / \Phi (G) $ is $ \tilde {\mathrm{F}} (G / \Phi (G)) $-subnormal.
  So $G/\Phi(G)$ is nilpotent. Now $G$ is nilpotent, a contradiction.

Assume that $ \Phi (G) = 1 $. Then   $ \tilde {\mathrm{F}} (G) = \mathrm{Soc} (G) $. Let $ M $ be an abnormal maximal
  subgroup of $G$. Since $M$ is $\tilde{\mathrm{F}}(G)$-subnormal maximal subgroup, we see that
 $M\trianglelefteq M\tilde{\mathrm{F}}(G)$. Hence $\tilde{\mathrm{F}}(G)\subseteq M$. It means that
 $\tilde{\mathrm{F}}(G)\leq\Delta(G)$, where $ \Delta(G)$ is the intersection of all abnormal maximal subgroups of $G$. From $\Phi(G)=1$ and Gasch\"{u}tz result (see \cite[1, Corollary 2.36]{Guo2015}) it follows that
 $\tilde{\mathrm{F}}(G)\subseteq \mathrm{Z}(G)$. Since $C_G(\tilde{\mathrm{F}}(G))\subseteq\tilde{\mathrm{F}}(G)$, we see that  $G\subseteq\tilde{\mathrm{F}}(G)\subseteq \mathrm{Z}(G)$. Thus $G$ is nilpotent,
 the  contradiction.

 $(2)$ Let $G$ be a group. Then $\mathrm{F}^*(\mathbf{f}(G))=\Phi(\mathbf{f}(G))$ (see the proof of Theorem \ref{len1}). Now all maximal subgroups of   $\mathbf{f}(G)$ are $\mathrm{F}^*(\mathbf{f}(G))$-subnormal. Hence $\mathbf{f}(G)\in\mathfrak{F}$. Since $\mathfrak{F}$ is Q-closed, we see that $G\simeq \mathbf{f}(G)/\mathrm{\tilde F}(\mathbf{f}(G))\in\mathfrak{F}$. Thus $\mathfrak{F}$ is the class of all groups.        \end{proof}



\begin{theorem}  The following statements for a group $G$ are equivalent:

$(1)$ $G$ is nilpotent;

$(2)$ Every abnormal subgroup of $ G $ is $ \mathrm{F}^{*}(G) $-subnormal
 subgroup of $ G $;

$(3)$ All normalizers  of Sylow subgroups of $G$ are  $ \mathrm{F}^{*} (G) $-subnormal;

$(4)$ All cyclic primary subgroups of $G$ are  $ \mathrm{F}^{*} (G) $-subnormal;

$(5)$ All  Sylow subgroups of $G$ are $ \mathrm{F}^{*} (G) $-subnormal.
\end{theorem}

  \begin{proof}  $(1)\Rightarrow(2)$. Let $G$ be a nilpotent group. It is clear that every
  subgroup of $G$ is $\mathrm{F}^*(G)$-subnormal. Hence (1) implies (2).

 $(2)\Rightarrow(3)$. It is well known that all normalizers of   Sylow subgroups are abnormal. Therefore (2) implies (3).

 $(3)\Rightarrow(4)$. Every cyclic primary subgroup $H$ of $G$ is contained in some Sylow
 subgroup $P$ of $G$. So
 $H\trianglelefteq\trianglelefteq P\trianglelefteq N_G(P)\trianglelefteq \trianglelefteq N_G(P)\mathrm{F}^*(G)$.
 Hence $H\trianglelefteq\trianglelefteq N_G(P)\mathrm{F}^*(G)$. In particular, $H$ is subnormal in $H\mathrm{F}^*(G)$.
 Hence (3) implies (4).

 $(4)\Rightarrow(5)$. Let $ P$ be a Sylow subgroup of $G$ and  $ x \in P $. Then
 $ \langle x \rangle $ is the $ \mathrm{F}^{*} (G) $-subnormal subgroup. So
 $ \langle x \rangle \trianglelefteq \trianglelefteq \langle x \rangle \mathrm{F}^{*} (G) $. Note
that $ \langle x \rangle \trianglelefteq \trianglelefteq P $. Since
$ \langle x \rangle \leq P \cap \langle x \rangle F^{*} (G) $, by
\cite[Theorem 1.1.7]{PFG} $ \langle x \rangle $ is the subnormal subgroup in the product
 $ P (\langle x \rangle \mathrm{F}^{*} (G))=P\mathrm{F}^*(G) $.
 Since $ P $ is generated by its cyclic subnormal in $ P \mathrm{F}^{*} (G) $ subgroups, we see
 that $ P \trianglelefteq \trianglelefteq P\mathrm{F}^{*} (G) $. So (4) implies (5).

$(5)\Rightarrow(1)$.
Let $P$ be a Sylow subgroup of $G$.  Since $P$ is pronormal subnormal subgroup of
  $P\mathrm{F}^*(G)$,  we see that $\mathrm{F}^*(G)\leq N_G(P)$ by \cite[I, Lemma 6.3(d)]{s8}. Now $\mathrm{F}^*(G)$ lies
 in the intersection of all normalizers of Sylow subgroups of $G$. So $ \mathrm{F}^{*} (G) \subseteq \mathrm{Z}_{\infty} (G)$ by \cite[5, Corollary 3]{Baer1953}.  Note that $G^\mathfrak{N}\leq C_G(\mathrm{Z}_{\infty} (G))$ by \cite[4, Claim (10)]{Baer1953}. From $C_G(\mathrm{F}^*(G))\subseteq \mathrm{F}^*(G)$, we see that $G^\mathfrak{N}\leq \mathrm{Z}_{\infty} (G)$. It means that  $G$ is nilpotent. \end{proof}

{\bibliographystyle{spmpsci}
\small
\parskip=-0mm
\parsep=0mm
\itemsep=0mm
 \bibliography{functorials}}

\end{document}